\documentclass[letterpaper,11pt,reqno]{amsart}

\usepackage{amsmath, amssymb, amscd, MnSymbol,mathrsfs}
\DeclareMathAlphabet{\mathbbm}{U}{bbm}{m}{n}

\newtheorem{Theorem}{Theorem}[section]
\newtheorem{prop}[Theorem]{Proposition}

\newtheorem{cor}[Theorem]{Corollary}

\theoremstyle{definition}
\newtheorem{definition}[Theorem]{Definition}

\newcommand{\impl}{\rightarrow}

\newcommand{\Int}{\mathbf{Int}}
\newcommand{\KF}{\mathbf{K4}}
\newcommand{\Ext}{\mathbf{Ext}}

\newcommand{\Frm}{\mathsf{Fm}}

\newcommand{\lbr}{\langle}
\newcommand{\rbr}{\rangle}

\newcommand{\Ax}{Ax} 
\newcommand{\Logl}{\mathcal{L}} 
\newcommand{\LogL}{\mathsf{L}}

\newcommand{\DS}{\mathsf{S}}

\newcommand{\Rules}{\mathsf{R}}

\newcommand{\Ds}{\DS \bydef \langle \Ax,\Rules \rangle}

\newcommand{\mpos}{\oplus}
\newcommand{\mrej}{\ominus}
\newcommand{\mnu}{\odot}
\newcommand{\notmnu}{\overline{\odot}}
\newcommand{\St}{\mathsf{St}}

\newcommand{\bydef} {:=}
\newcommand{\means}{\leftrightharpoons}

\newcommand{\Inf}{\mathfrak{I}}

\newcommand{\ulog}{\Logl = \langle \LogL^+,\LogL^- \rangle}

\newcommand{\set}[2]{\{#1 \mid #2\}}

\def\alg#1{\mathsf{#1}}
\def\Alg#1{\mathscr{#1}}

\title{{\L}-axiomatizability in Intermediate and Normal Modal Logics}

\subjclass{This is not required.}

\author{Alex Citkin}

\address{metropolitan Telecommunications, USA} 

\email{acitkin@gmail.com}

\begin{document}

\begin{abstract}
A set $F$ of formulas is complete relative to a given class of logics, if every logic from this class can be axiomatized by formulas from $F$. A set of formulas  $F$  is {\L}-complete relative to a given class of logics, if every logic of this class can be {\L}-axiomatized by formulas from $F$, that is, every of these logics can be defined by an $\L$-deductive system with axioms and anti-axioms from $F$ and inference rules modus ponens, modus tollens, substitution and reverse substitution. We prove that every complete relative to $\Ext\Int$ (or $\Ext\KF$) set of formulas is {\L}-complete. In particular, every logic from $\Ext\Int$ (or $\Ext\KF$) can be {\L}-axiomatized by Zakharyaschev's canonical formulas.
\end{abstract} 

\maketitle


\section{Introduction}

Canonical formulas were introduced by M.~Zakharyaschev (for details and references see  \cite{Chagrov_Zakh}). They have been instrumental in studying intermediate and normal modal logic. The canonical formulas form a complete set of formulas, meaning that any intermediate logic or any normal extension of $\KF$ can be axiomatized over intuitionistic propositional calculus (IPC) or, respectively, over $\KF$ by canonical formulas. Our goal is to demonstrate that canonical formulas form the complete set not only for proving formulas but also for deriving, while using a {\L}ukasiewicz-style calculi ({\L}-deductive system), the rejection of formulas. We will prove a stronger statement: one can construct {\L}-axiomatization of every logic from $\Ext\Int$ or $\Ext\KF$ using any given complete set of formulas. 

The refutation system for various intermediate and normal modal logics were extensively studied by T.~Skura, V.~Goranko (see, for instance, \cite{Skura1998,Goranko_Refutation_1994}). In \cite{Skura_Syntactic_1994} T.~Skura observed that in case of finitely approximated logics the Jankov formulas give the complete set of anti-axioms, that is the additional axioms that can be used to prove refutation of a formula. The canonical formulas are, in a way, the modified Jankov or, more precisely, frame formulas. It turned out that we can effectively use the canonical formulas for refutation.

In Section \ref{ref}, we give the background information regarding {\L}-deductive systems. In Section \ref{refInt}, we prove that every intermediate logic has an {\L}-deductive system defining it and having axioms from a given complete set of formulas. And in Section \ref{refK4}, we extend this result to the normal extensions of $\KF$.

\section{{\L}-Deductive Systems} \label{ref}

\subsection{Refutation Systems}

Commonly, we use a deducting system in order to prove a formula and we use semantical means in order to disprove a formula. But the rejection of a formula can also be established syntactically. For instance, by Modus Tollens we can derive that a formula $A$ is refutable if we prove $A \to B$ and disprove $B$.    

The idea, to include the rejected propositions into proofs can be traced back to R.~Carnap \footnote{And in traditional logic even to Aristotle and the Stoics.} \cite{Carnap_Formalization_1943}. But J.~{L}ukasiewicz was the first who constructed a deductive system for proving refutability \cite{Lukasiewicz_Book}.

In general, there are two ways of handling the refutation syntactically: direct and indirect. To determine weather a formula $A$ is refutable we can do one of the following

\begin{itemize}
\item derive in a meta-logic a statement about refutability of $A$ ({\L}-proof - {\L}ukasiewicz-style proof) 
\item derive from $A$ a formula $B$ that we already know is refutable (an anti-axiom) and then apply Modus Tollens (i-proof - indirect proof, Carnap's way)
\end{itemize} 

An existence of an {\L}-proof entails the existence of i-proof. The converse is true under some assumptions (some weak form of the deduction theorem \cite{Staszek_1971}).  

\subsubsection{Examples of i-complete systems} Let $\Frm$ be a set of (propositional) formulas and $\Sigma$ be a set of all simultaneous substitutions of formulas for (propositional) variables. Let $\vdash$ be a structural consequence relation, that is, for any finite set of formula $\Gamma \subseteq \Frm$ and any formulas $A,B \in \Frm$  the following holds
\begin{itemize}
\item[(a)] $A \vdash A$
\item[(b)] if $\Gamma \vdash A$, then $\Gamma, B \vdash A$
\item[(c)] if $\Gamma,A \vdash B$ and $\Gamma \vdash A$, then $\Gamma \vdash B$
\item[(d)] if $\Gamma \vdash A$, then $\sigma(\Gamma) \vdash \sigma(A)$ 
\end{itemize}

Given a consequence relation $\vdash$, we say that a pair of sets of formulas $\lbr \Ax^+; \Ax^- \rbr$ is an i-complete system for $\vdash$ if
\[
\vdash A \text{ if and only if } 
\]

\subsection{Definitions}

By $\Frm$ we denote the set of all (propositional) formulas in a given language containing $\impl$ among connectives.

A \textit{logic} is a subset $\LogL \subseteq \Frm$ closed under rules Modus Ponens and Substitution, i.e. for any $A,B \in \Frm$ and any $\sigma \in \Sigma$
\[
A, (A \impl B) \in \LogL \text{ entails } B \in \LogL \text{ and } \sigma(A) \in \LogL.
\]

We will assume that there is a class of models (algebras, matrices, etc.) $\Alg{M}$ and for every formula $A \in \Frm$ it is defined whether $A$ is valid in a given model $\alg{M}$ (in written $\alg{M} \vDash A$), or not (in written $\alg{M} \nvDash A$).

\subsection{{\L}-deductive Systems}

If $A \in \Frm$ is a formulas, than $\mpos A$ and $\mrej A$ are (atomic) statements. $\mpos A$ is a \textit{positive statement} (assertion) and $\mrej A$ is a \textit{negative statement} (rejection).  The set of all positive statements we denote by $\St^+$, the set of all negative statements we denote by $\St^-$, and $\St$ denotes the set of all statements, that is, $\St \bydef \St^+ \cup \St^-$. 

By {\L}-deductive system we understand a couple $\Ds$, where $\Ax \bydef \Ax^+ \cup \Ax^-$ and  $\Ax^+ \subseteq \St^+$ is a set of \textit{axioms}, $\Ax^- \subseteq \St^-$ is a set of \textit{anti-axioms}, and $\Rules$ is a set of the following rules:
\[
\begin{array}{lll}
\text{Modus Ponens} & \mpos(A \impl B),\mpos A / \mpos B & (MP)\\
\text{Substitution} &\mpos A/ \mpos \sigma(A), \text{ for all } \sigma \in \Sigma & (Sb)\\
\text{Modus Tolens} &\mpos(A \impl B),\mrej B / \mrej A & (MT)\\
\text{Reverse Substitution} &\mrej \sigma(A)/ \mrej A, \text{ for all } \sigma \in \Sigma & (RS)\\
\end{array}
\]

\subsection{{\L}-Inference}

In a natural way, every deduction system $\Ds$ defines an inference: if $\Gamma$ is a set of statements and $\alpha$ is a statement, a sequence $\alpha_1,\dots,\alpha_n$ of statements is an \textit{inference} of $\alpha$ from $\Gamma$ if $\alpha_n$ is $\alpha$ and for each $i \in \{1,\dots,n\}$ one of the following holds
\begin{itemize}
\item[(a)]  $\alpha_i \in \Ax$
\item[(b)]  $\alpha_i \in \Gamma$
\item[(c)] $\alpha_i$ can be obtained from the preceding statements by one of the rules.
\end{itemize}
If there exists an inference of $\alpha$ from $\Gamma$, we say that $\alpha$ is \textit{derivable in} $\DS$ \textit{from} $\Gamma$, and we denote this by $\Gamma \vdash_\DS \alpha$ (and we will omit index if no confusion arises). The \textit{length of an inference} is a number of statements in it. 

If $\mnu \in \{\mpos, \mrej\}$, then $\notmnu A$ is a statement with the sign opposite to $\mnu$, that is, if $\mnu = \mpos$, then $\notmnu = \mrej$ and if $\mnu = \mrej$, then $\notmnu = \mpos$.

\begin{prop} \label{posinfer} For any {\L}-deductive system $\DS$, if $\Gamma \subseteq \St^+$, $\alpha \in \St^+$ and $\Inf \bydef \alpha_1,\dots,\alpha_n,\alpha$ is an inference of $\alpha$ from $\Gamma$, then, omitting from $\Inf$ all negative statements, the obtained sequence $\Inf^+$ still will be an inference of $\alpha$ from $\Gamma$.
\end{prop}
\begin{proof} Proof by induction on the length of $\Inf$.

\textbf{Basis}. If $\Inf$ contain a single statement $\alpha$, the inference already consists of only positive statements. 

\textbf{Inductive Hypothesis.} Assume that for all inferences of the length at most $m$ the statement is true.

\textbf{Step.} Let $\Inf$ be an inference of $\alpha$ from $\Gamma$ of the length $m+1$. By the definition of inference, either $\alpha \in \Ax \cup \Gamma$, or $\alpha$ is obtained by some rule from the preceding members of $\Inf$. If $\alpha \in \Ax$ or $\alpha \in \Gamma$, then the single-element sequence $\alpha$ is an inference. 

Suppose $\alpha$ is obtained by one of the rules. Since $\alpha$ is a positive statement, it can be obtained only by (MP) or (Sb). Let us consider these to cases.

\textit{A Case of (MP).} Let $\Inf \bydef \alpha_1,\dots,\alpha_m,\alpha$. Suppose $\alpha$ is obtained by (MP) and $\alpha = \mpos A$ for some $A \in \Frm$. Then, for some formula $B \in \Frm$, the statements $\mpos(B \impl A)$ and $\mpos B$ occur in $\Inf$. Assume $\mpos B = \alpha_i$ and $\mpos (B \impl A) = \alpha_j$ members of $\Inf$. Let $1 \leq k \leq m$ be the greatest index such that $\alpha_k \in \Inf$ and $\alpha_k$ is a positive statement (that is, all statements $\alpha_{k+1},\dots,\alpha_m$ are negative). Clearly, $1 \leq i,j \leq k$. Then, the first $k$ elements $\Inf$ form an inference $\Inf_k$ and $\Inf_k$ contains both of statements $\mpos B$ and $\mpos (B \impl A)$. By the inductive hypothesis, we can omit in $\Inf_k$ all negative statements and obtain a new inference $\Inf_k^+$ that contains only positive statements. It is easy to see that the statements $\mpos B$ and $\mpos (A \impl B)$ are members of $\Inf^+_k$. Hence, we can add to $\Inf^+_k$ the statement $\mpos A$ and obtain an inference of $\alpha$ from $\Gamma$. Note, that obtained inference is exactly an inference obtained from $\Inf$ by omitting all negative statements.    

\textit{A Case of (Sb).}  This case can be considered in the way similar to the case of (MP).
\end{proof}

\begin{cor}  For any {\L}-deductive system $\DS \bydef \lbr \Ax, \Rules \rbr$, if $\Gamma \subseteq \St^+$, $\alpha \in \St^+$ and $\Gamma \vdash_\DS \mpos \alpha$, then there is an inference of $\alpha$ from $\Gamma$ containing only the positive statements.
\end{cor}

\subsection{Coherent and Full  {\L}-deductive Systems}

\begin{definition} {\L}-deductive system $\Ds$ we will call \textit{coherent} if for no $A \in \Frm$ 
\[
\vdash_\DS \mpos A \text{ and } \vdash_\DS \mrej A.
\] 
And we will call $\DS$ \textit{full} if for every $A \in \Frm$
\[
\vdash_\DS \mpos A \text{ or } \vdash_\DS \mrej A.
\]
A coherent and full system will be called \textit{standard}.
\end{definition}

If $A \in \Frm$ is a formula and $\alg{M}$ is a model we let
\begin{equation}
\alg{M} \vDash \mpos A \means \alg{M} \vDash A \text{ and } \alg{M} \vDash \mrej A \means \alg{M} \nvDash A.
\end{equation}
If $\alg{M} \vDash \mnu A$ we say that the statement $\mnu A$ \textit{is valid in} $\alg{M}$.

We say that a model $\alg{M}$ is an \textit{adequate regular model} for an {\L}-deductive system $\DS$, if for every $A \in \Frm$
\begin{equation}
\vdash_\DS \mnu A \text{ if and only if } \alg{M} \vDash \mnu A.
\end{equation}
 
It is not hard to see that the following holds.

\begin{prop}\label{standard}
If a given {\L}-deductive system $\DS$ has an adequate regular model, then the system $\DS$ is standard.
\end{prop}
In this paper, we consider only regular models.

Let us also observe that in order to prove that a model $\alg{M}$ is adequate for a given {\L}-deductive system $\Ds$ as long as all axioms and anti-axioms are valid in $\alg{M}$.

\begin{prop} \label{adequate} Let $\alg{M}$ be a model and $\Ds$ be an {\L}-deductive system. If 
\[
\text{ for every } A \in \Ax, \alg{M} \vDash \mnu A,
\] 
then $\alg{M}$ is adequate for $\DS$.
\end{prop}
\begin{proof} The proof can be done by induction on the length of inference. Indeed, all rules preserve the validity of statements, i.e. if the premisses of a rule are valid in $\alg{M}$, then the conclusion is valid too. 
\end{proof}

\subsection{Logics Defined by {\L}-deductive Systems}

Every given deductive system $\DS$ defines the pair 
\[
\lbr \LogL^+,\LogL^- \rbr \text{, where } \LogL^+ \bydef \set{A \in \Frm}{\vdash_\DS \mpos A} \text{ and } \LogL^- \bydef \set{A \in \Frm}{\vdash_\DS \mrej A}
\]
that we call a \textit{logic}. The logic defined by a given {\L}-deductive system $\DS$ we will denote by $\Logl(\DS)$. 

We say that a logic $\ulog$ is \textit{coherent}, \textit{full} or \textit{standard} if the defining {\L}-deductive system is coherent, full or, respectively, standard. It is easy to see that a logic $\Logl$  is coherent if and only if $\LogL^+ \cap \LogL^- = \emptyset$; logic $\Logl$ is full if and only if $\LogL^+ \cup \LogL^- = \Frm$; and logic $\Logl$ is standard if and only if $\LogL^- = \Frm \setminus \LogL^+$. 

A logic is said to be \textit{finitely {\L}-axiomatizable} if it can be defined by an {\L}-deductive system with the finite set of axioms.

Any pair $\ulog$, where $\LogL^+,\LogL^- \subseteq \Frm$ and $\LogL^+$ is closed under (MP) and (SB) and $\LogL^-$ is closed under (MT) and (RS), is a logic. Indeed, $\ulog$ can be defined by an {\L}-deductive system system in which
\[
\Ax = \set{\mpos A}{A \in \LogL^+} \cup \set{\mrej A}{A \in \LogL^-}.
\]

Recall that a couple $\alg{M} (\Logl) \bydef \lbr \Frm, \LogL^+\rbr$, where $\LogL^+$ is a set of designated values,  is a \textit{Lindenbaum matrix} of a logic $\ulog$.

\begin{prop}
If $\Ds$ is a standard {\L}-deductive system, the Lindenbaum matrix of $\alg{M}(\Logl(\DS))$ is an adequate model of $\DS$.
\end{prop}
\begin{proof} Due to Proposition \ref{adequate} it suffices to check that all axioms of $\DS$ are valid in $\alg{M}(\Logl(\DS))$. 

Suppose $\mpos A \in \Ax$. Then $A \in \LogL^+$. Recall that $\LogL^+$ is closed under substitutions. Hence, $\alg{M}(\Logl(\DS)) \vDash \mpos A$.  

Suppose $\mrej A \in \Ax$. Then $A \in \LogL^-$. Hence, $\alg{M}(\Logl(\DS)) \vDash \mrej A$.
\end{proof}

\subsection{The Theorem about Symmetry in \textbf{ExtInt}}

From this point forward we consider only the deductive systems $\DS$ in which $\vdash_\DS \mpos(A \impl (B \impl A))$ for all $A,B \in \Frm$.

The meaning of the following theorem is very straightforward: if we cannot derive a formula $A$ in a given regular deductive system, but can derive it from the some set of formulas $\Gamma$, then $\Gamma$ contains a formula $B$ not derivable in the system and, moreover, $\mrej B$ and be {\L}-derived from $\mrej A$. In a way, the following theorem can be regarded as a strengthening of Modus Tollens.   

\begin{Theorem}[about symmetry in \textbf{ExtInt}] \label{SymTh} For any {\L}-deductive system $\Ds$ and any $A_1,\dots,A_n,B \in  \Frm$ if
\[
\nvdash_\DS \mpos B \text{ and } \mpos A_1,\dots,\mpos A_n \vdash_\DS \mpos B,
\] 
then 
\[
\mrej B \vdash_\DS \mrej A_i
\] 
for some $1 \leq i \leq n$.
\end{Theorem}
\begin{proof}
We will prove the claim by induction on the length of inference of $\mpos B$ from $\mpos A_1,\dots,\mpos A_n$. By virtue of the Proposition \ref{posinfer} we can safely assume that the inference consists of only positive statements.

\textbf{Basis.} Suppose there is an inference of $\mpos B$ from $\mpos A_1,\dots,\mpos A_n$ of the length 1. Then, by the definition of inference, $\mpos B = \mpos A_i$ for some $1 \leq i \leq n$, for $\mpos B \notin \Ax$, due to $\nvdash_\DS \mpos B$. Hence, $\mrej B \vdash_\DS \mrej A_i$.

\textbf{Inductive Hypothesis.} Assume that if there is an inference of the length at most $m$ of $\mpos B$ from $\mpos A_1,\dots \mpos A_n$, then $\mrej B \vdash_\DS \mrej A_i$ for some $1 \leq i \leq n$.

\textbf{Inductive Step.} Let $\mpos B_1,\dots,\mpos B_m,\mpos B$ be an inference of $\mpos B$ from $\mpos A_1,\dots,\mpos A_n$. The cases (a) and (b) from the definition of inference can were considered in the basis of induction. Let us assume that the statement $\mpos B$ is obtained by one of the rules. Due to this statement is positive, it can be obtained only by (MP) or (Sb).

\textbf{The case of (MP).} Suppose $B_j = \mpos (C \impl B)$ and $B_k = \mpos C$, where $1 \leq j,k \leq m$. There are two possible subcases: 
\begin{itemize}
\item[(a)] $\vdash_\DS \mpos (C \impl B)$ ;
\item[(b)] $\nvdash_\DS \mpos (C \impl B)$. 
\end{itemize}

\textbf{Subcase (a).} Suppose $\vdash_\DS \mpos (C \impl B)$. Then, $\nvdash_\DS \mpos C$, for $\nvdash_\DS \mpos B$. Note, that the sequence $\mpos B_1, \dots,\mpos B_k$ is an inference of $\mpos C$ from $\mpos A_1,\dots,\mpos A_n$ and $1 \leq k \leq m$. Hence,by the induction hypothesis, 
\begin{equation}
\mrej C \vdash_\DS \mrej A_i, \text{ for some } 1 \leq i \leq n. \label{suba-1}
\end{equation}
On the other hand, we can apply (MT) to $\vdash_\DS \mpos (C \impl B)$ and $\mrej B$ and obtain
\begin{equation}
\mrej B \vdash_\DS \mrej C. \label{suba-2}
\end{equation}
And from \eqref{suba-1} and \eqref{suba-2} we can derive
\begin{equation}
\mrej B \vdash_\DS \mrej A_i, \text{ for some } 1 \leq i \leq n. 
\end{equation}

\textbf{Subcase (b).} Suppose $\nvdash_\DS \mpos (C \impl B)$. Then, we observe that $\mpos B_1,\dots, \mpos B_j$ is an inference of $\mpos (C \impl B)$ from $\mpos A_1,\dots,\mpos A_n$ and $1 \leq i \leq m$. So, we can apply the induction hypothesis and get
\begin{equation}
\mrej(C \impl B) \vdash_\DS \mrej A_i \text{ for some } 1 \leq i \leq n. \label{subb-1}
\end{equation}
On the other hand, we can apply (MT) to $\vdash \mpos (B \impl (C \impl B))$ and $\mrej B$ and obtain
\begin{equation}
\mrej B \vdash_\DS \mrej (C \impl B).  \label{subb-2}
\end{equation}
And from \eqref{subb-1} and \eqref{subb-2} we can derive
\begin{equation}
\mrej B \vdash_\DS \mrej A_i, \text{ for some } 1 \leq i \leq n. 
\end{equation}

\textbf{The case of (Rs).} Suppose $B = \sigma(B_j)$, where $1 \leq j \leq m$. Then $\nvdash_\DS \mpos B_j$, for $\nvdash_\DS \mpos B$. Also, note that $\mpos B_1,\dots,\mpos B_j$ is an inference of $B_j$ from $\mpos A_1,\dots,\mpos A_n$ and $1 \leq j \leq m$. Hence, by the induction hypothesis,
\begin{equation} 
\mrej B_j \vdash_\DS \mrej A_i \text{ for some } 1\leq i \leq n. \label{caseRS-1} 
\end{equation}
On the other hand, $\mrej B = \mrej \sigma(B_j)$ and from $\mrej \sigma(B_j)$, by (RS), we have
\begin{equation}
\mrej B \vdash_\DS \mrej B_j. \label{caseRS-2}
\end{equation}
From \eqref{caseRS-1} and \eqref{caseRS-2} we have

\begin{equation} 
\mrej B \vdash_\DS \mrej A_i \text{ for some } 1\leq i \leq n. 
\end{equation}

\end{proof}

\section{Refutation in Ext\textbf{Int}} \label{refInt}

If $\Gamma \subseteq \Frm$ and $A \in \Frm$, then by $\Gamma \Vdash A$ we denote that $A$ is derivable from $\Gamma$ in Intuitionistic Propositional Calculus (IPC) with substitution (e.g, \cite[Section 7.1.3]{Heyting_Book_1966}). $\Int + \Gamma$ will denote a logic axiomatized over $\Int$ by $\Gamma$, that is $\Int + \Gamma \bydef \set{A \in \Frm}{\Gamma \Vdash A}$. And $\Gamma + A$ means the same as $\Gamma + \{A\}$.

A set $\alg{F}$ of formulas is said to be \textit{complete} \cite{Zakharyashchev_Syntax_1988} (or \textit{sufficiently rich} \cite{Tomaszewski_PhD}) if every logic from $\Ext\Int$ can be axiomatized over $\Int$ by some formulas from $\alg{F}$. An obvious characterization of completeness can be given by the following Proposition:

\begin{prop} \label{interder} A set of formulas $\alg{F}$ is complete if and only if for each formula $A$ such that $\Int \not\Vdash A$ there are formulas $A_1,\dots,A_n \in \alg{F}$ and
\begin{equation}
A_1,\dots,A_n \Vdash A \text{ and } A \Vdash A_i \text{ for all } i=1,\dots,n. \label{compl}
\end{equation} 
\end{prop}
\begin{proof} Clearly, if \eqref{compl} holds, every logic from $\Ext\Int$ can be axiomatized over $\Int$ by some formulas from $\alg{F}$.

Conversely, if $\alg{F}$ is a complete set, we can consider a logic $\LogL \bydef \Int + A$ axiomatized over $\Int$ by formula $A$. By the definition of completeness, for some $A_1.\dots,A_n \in \alg{F}$ we have $\LogL = \Int + \{A_1,\dots,A_n\}$, from which \eqref{compl} immediately follows.   
\end{proof}

Perhaps, the best known complete set of formulas is a set of canonical formulas introduced by M.~Zakharyaschev (cf. \cite{Chagrov_Zakh} for definitions, references and history). For our purposes it is important only that canonical formulas satisfy  \eqref{compl}  (cf. \cite[Theorem 9.44(i)]{Chagrov_Zakh}) and, thus, they form a complete set. 

By $\mathsf{IPL}$ we denote the intuitionistic propositional logic, that is, $\mathsf{IPL} \bydef\set{A \in \Frm}{\Vdash A}$. 

We say that $\ulog$ is a \textit{standard intermediate logic} if $\mathsf{IPL} \subseteq \LogL^+ \subset \Frm$, and $\LogL$ is closed under rules Modus Ponens and Substitution. 

\subsection{Completeness Theorem}

By $\Ax^i$ we will denote the set of positive statements obtained from the axioms of IPC. And by $\Frm^c$ we denote a given complete set of all formulas (for instance, a set of all canonical formulas). 

Let us note the following.

\begin{prop}\label{infconv} Assume $A_1,\dots,A_n,B \in \Frm$ and $\Ds$ is such an {\L}-deductive system that $\Ax^i \subseteq \Ax$. Then
\[
A_1,\dots,A_n \Vdash B, \text{ entails } \mpos A_1, \dots, \mpos A_n \vdash_\DS \mpos B.
\]
\end{prop}
\begin{proof} It is not hard to see that any inference of $B$ from $A_1,\dots,A_n$ in IPC can be easily converted into an inference of $\mpos B$ from $\mpos A_1,\dots,\mpos A_n$ in $\DS$.
\end{proof}

\begin{Theorem} \label{ThCompl} Every intermediate logic $\Logl$ can be defined by a standard deductive system $\Ds$, where every axiom $\mnu A \in \Ax$ is as statement obtained either from an axiom of IPC or from a formula $A \in \Frm^c$. In other words, given a complete set of formulas $\Frm^c$, every intermediate logic can be {\L}-axiomatized over IPC by formulas from $\Frm^c$ as additional axioms and anti-axioms. 
\end{Theorem}
\begin{proof}
Let $\ulog$ be an intermediate logic. Let us consider the {\L}-deductive system $\Ds$, where 
\begin{equation}
\Ax = \Ax^i \cup \set{\mpos A}{A \in \LogL^+ \cap \Frm^c} \cup \set{\mrej A}{A \in \LogL^- \cap \Frm^c},
\end{equation}
i.e. axioms of $\DS$ are statements obtained from the axioms of IPC and canonical formulas. We need to demonstrate that $DS$ defines $\Logl$. We will show
\begin{itemize}
\item[(a)] If $A \in \LogL^+$, then $\vdash_\DS \mpos A$;
\item[(b)] If $A \in \LogL^-$, then $\vdash_\DS \mrej A$;
\item[(c)] $\DS$ is coherent.  
\end{itemize}

Note, that fullness of $\DS$ immediately follows from (a) and (b). Thus, if $\DS$ enjoys (a),(b) and (c), then $\DS$ is standard. Also, it is not hard to see, that if $\DS$ enjoys (a),(b) and (c), then $\DS$ defines $\Logl$. So, all we need to do is to prove (a),(b) and (c). 

First, we will establish coherence of the system $\DS$.

\textbf{Proof of (c).} Let us take a Lindenbaum matrix $\alg{M}(\Logl) \bydef \lbr \Frm, \LogL^+ \rbr$. By the definition of $\DS$ all the axioms of $\DS$ are valid in $\alg{M}(\Logl)$. Hence, by the Proposition \ref{adequate}, $\alg{M}(\Logl)$ is an adequate model of $\DS$ and, by virtue of the Proposition \ref{standard}, $\DS$ is a standards {\L}-deductive system and, thus, is coherent.

\textbf{Proof of (a).} Assume $A \in \LogL^+$. If $A$ is derivable in IPC, that is $\Vdash A$, then, by virtue of the Proposition \ref{infconv}, 
\[
\vdash_\DS \mpos A.
\]

Assume $A \in \LogL^+$ and $A$ is not derivable in IPC. Then, by virtue of \eqref{compl}, there are such formulas $C_1,\dots,C_n \in \Frm^c$ that
\[
C_1, \dots, C_n \Vdash A.
\]
Then, by virtue of the Proposition \ref{infconv},
\[
\mpos C_1, \dots, \mpos C_n \vdash_\DS \mpos A.
\]
Recall, that by the definition of $\DS$, we have $\mpos C_1, \dots, \mpos C_n \in \Ax$. Hence,
\[
\vdash_\DS \mpos A.
\] 

\textbf{Proof of (b).} Assume $A \in \LogL^-$. Then, by virtue of \eqref{compl}, there are such formulas $C_1,\dots,C_n \in \Frm^c$ that
\begin{equation}
C_1,\dots,C_n \Vdash A \text{ and } A \Vdash C_i \text{ for all } i=1,\dots,n. \label{Zakh}
\end{equation}
Let us observe that, due to $A \in \LogL^-$, one of the formulas $C_i, i=1,\dots,n $ is in $\LogL^-$. Suppose $C_1 \in \LogL^-$ and, hence, 
\begin{equation}
\mrej C_1 \in \Ax. \label{sym0} 
\end{equation}
We already proved that system $\DS$ is coherent. Thus,
\begin{equation} 
\nvdash_\DS \mpos C_1. \label{sym1}
\end{equation} 
On the other hand, by \ref{Zakh},
\begin{equation}
A \Vdash C_1. \label{sym2}
\end{equation} 
And, by virtue of the Proposition \ref{infconv}, 
\begin{equation}
\mpos A \vdash_\DS \mpos C_1. \label{sym3}
\end{equation} 
From \eqref{sym1} and \eqref{sym3}, by virtue of the Theorem \ref{SymTh}, we have
\begin{equation}
\mrej C_1 \vdash_\DS \mrej A. \label{sym4}
\end{equation}
And from \eqref{sym0} and \eqref{sym4} we can conclude 
\[
\vdash_\DS \mrej A.
\]
\end{proof}

\begin{cor} Every finitely {\L}-axiomatizable intermediate logic $\Logl$ can be defined by a standard deductive system $\Ds$ with finite number of axioms and every axiom $\mnu A \in \Ax$ is as statement obtained from an axiom of IPC or from a canonical formula $A$. 
\end{cor}
\begin{proof} The proof immediately follows from the above Theorem and the finitarity of the relation $\vdash_\DS$. 
\end{proof}

\section{Refutation in \textbf{NExtK4}} \label{refK4}

From this point forward, $\Frm$ will denote the set of all formulas in the signature $\land,\lor,\impl,\sim,\Box$.  

In order to consider normal modal logics, first, we need to extend the set $\Rules$ of rules by adding to (MP),(MT),(Sb) and (RS) the rules
\[
\begin{array}{lll}
\text{Necessitation} & \mpos(A),\mpos \Box A& (NS)\\
\text{Reverse Necessitation} &\mrej \Box A/ \mrej A & (RN)\\
\end{array}
\]

Next, we need to establish that the Theorem about symmetry holds in \textbf{NExtK4}.

\begin{Theorem}[about symmetry in \textbf{NExtK4}] \label{SymThK4} For any {\L}-deductive system $\Ds$ and any $A_1,\dots,A_n,B \in  \Frm$ if
\[
\nvdash_\DS \mpos B \text{ and } \mpos A_1,\dots,\mpos A_n \vdash_\DS \mpos B,
\] 
then 
\[
\mrej B \vdash_\DS \mrej A_i
\] 
for some $1 \leq i \leq n$.
\end{Theorem}
\begin{proof} Similarly to Theorem \ref{SymTh}, the proof can be done by induction on the length of inference. We can repeat the proof of the Theorem \ref{SymThK4} and consider only the additional case for (NS). 

\textbf{The case of (NS).} Suppose $B = \mpos \Box A_j$, where $1 \leq j < m$. Then $\nvdash_\DS \mpos A_j$, for $\nvdash_\DS \mpos \Box A_j$.  Note, that $\mpos A_1,\dots, \mpos A_j$ is an inference of $\mpos A_j$ from $\mpos A_1,\dots,\mpos A_{j-1}$. Hence, by the induction hypothesis,
\begin{equation}
\mrej A_j \vdash_\DS \mrej A_i, \text{ for some } 1 \leq i < j. \label{ns1}
\end{equation} 
By (RN), we also have
\begin{equation}
\mrej \Box A_j \vdash_\DS \mrej A_j. \label{ns2}
\end{equation}
And, combining \eqref{ns1} and \eqref{ns2}, we have
\[
\mrej \Box A_j \vdash_\DS \mrej A_i. 
\]
Recall, that $B = \Box A_j$, thus, we can conclude the proof of this case.
\end{proof}

Given a complete set of formulas, for instance, the set of canonical formulas, one can prove the following theorem (using the same argument as in proof of the Theorem \ref{ThCompl}).

\begin{Theorem} Every logic $\Logl \in \text{\normalfont \textbf{NExtK4}}$ can be defined by a standard deductive system $\Ds$, where every axiom $\mnu A \in \Ax$ is as statement obtained from an axiom of \textbf{\normalfont $\KF$} or from a canonical formula $A$. In other words, every logic from \textbf{\normalfont \textbf{NExtK4}} can be {\L}-axiomatized by canonical formulas as additional axioms. 
\end{Theorem}

And, similarly to intermediate logics, the following holds.

\begin{cor} Every finitely {\L}-axiomatizable logic $\Logl \in \text{\normalfont \textbf{NExtK4}}$ can be defined by a standard deductive system $\Ds$ with finite number of axioms and every axiom $\mnu A \in \Ax$ is as statement obtained from an axiom of \textbf{\normalfont $\KF$} or from a canonical formula $A$. 
\end{cor}

\bibliographystyle{acm}

\def\cprime{$'$}

\end{document}